\documentclass[11pt,a4paper]{amsart}
\usepackage{amsmath,amsxtra,amsfonts,amssymb,
latexsym,amsthm,amscd,verbatim,amstext,enumerate,graphicx}
\usepackage[mathscr]{eucal}
\makeatletter
\@namedef{subjclassname@2010}{%
  \textup{2010} Mathematics Subject Classification}
\makeatother

\newcommand{\im}{\mathrm{i}}

\newcommand{\Tr}{\mathrm{Tr}}
\newcommand{\diag}{\mathrm{diag}}

\theoremstyle{plain}
\newtheorem{theorem}{Theorem}
\newtheorem{proposition}{Proposition}

\newtheorem{corollary}{Corollary}
\theoremstyle{definition}
\newtheorem{definition}{Definition}

\numberwithin{equation}{section}

\begin{document}


\baselineskip=17pt



\title[]{Elliptic solutions to nonsymmetric Monge-Amp\`{e}re type equations I. The $d$-concavity and the comparison principle}

\author[Ha Tien Ngoan]{Ha Tien Ngoan}
\address{Institute of Mathematics\\ Vietnam Academy of Science and Technology\\
18 Hoang Quoc Viet, 10307  Cau Giay, Hanoi, Vietnam.}
\email{htngoan@math.ac.vn}

\author[Thai Thi Kim Chung]{Thai Thi Kim Chung}
\address{University of Transport Technology\\ 54 Trieu Khuc, Thanh Xuan, Hanoi, Vietnam}
\email{chungttk@utt.edu.vn}

\date{}

\begin{abstract}
We introduce the so-called $d$-concavity, $d \geq 0,$ and prove that the nonsymmetric Monge-Amp\`{e}re type function of matrix variable is concave in an appropriate unbounded and convex set. We prove also the comparison principle for nonsymmetric Monge-Amp\`{e}re type equations in the case when they are so-called $\delta$-elliptic with respect to compared functions with $0 \leq \delta <1$. 
\end{abstract}



\maketitle

\section{Introduction}
In this paper we consider the following nonsymmetric Monge-Amp\`ere type equations:
\begin{equation}\label{1}
\det \left [ D^2u -A(x,u,Du)-B(x,u,Du) \right ]=f(x,u,Du),\,\, x\in \Omega, \\
\end{equation}
where $\Omega$ is a bounded domain in $n$ dimensional Euclidean space $\mathbb{R}^n$ with smooth boundary, $Du$ and $D^2 u$ denote the gradient vector and the Hessian matrix of the second order derivatives of the function $u: \Omega \to \mathbb{R},$ respectively, $A$ is a given $n \times n$ symmetric matrix function defined on $\Omega \times \mathbb{R} \times \mathbb{R}^n$, $B$ is a given $n \times n$ skew-symmetric matrix function defined on $\Omega \times \mathbb{R} \times \mathbb{R}^n$, $f$ is a positive scalar valued function defined on $\Omega \times \mathbb{R} \times \mathbb{R}^n.$ As usual, we use $x,z,p,r$ to denote points in $\Omega, \mathbb{R}, \mathbb{R}^{ n }, \mathbb{R}^{n \times n},$ respectively.

\indent In the case that $B(x,z,p) \equiv 0$, equation \eqref{1} becomes
\begin{equation}\label{2}
\det [ D^2u -A(x,u,Du)]=f(x,u,Du),\,\,  x \in \Omega.
\end{equation}
For functions $u(x) \in C^2 (\Omega),$ we set 
\begin{equation}\label{3}
\omega (x,u) \equiv D^2 u(x)-A(x,u(x),Du(x)).
\end{equation}
We recall that the equation \eqref{1} or \eqref{2} is {\it elliptic with respect to function $u(x) \in C^2 (\Omega)$} whenever
\begin{equation*}
\lambda_{\min} (\omega(x,u))>0,\ \forall x \in \Omega.
\end{equation*}
Here and in what follows, we denote by $\lambda_{\min} (M)$ the smallest eigenvalue of a symmetric matrix $M \in \mathbb{R}^{n \times n}.$

For the Dirichlet problem for equation \eqref{2}, the existence of elliptic solutions was settled in \cite{4}, \cite{5}, \cite{6} by the method of continuity. In this method, the solvability of the Dirichlet problem is reduced to the establishment of $C^{2,\alpha}(\overline{\Omega})$ estimates for its elliptic solutions. It is well-known that the concavity of the following function
$$ F(\omega) = \log (\det \omega ),$$
considered as a function on the set of symmetric positive definite matrices $\omega = [\omega_{ij}]_{n \times n},$ has one of essential roles in establishing these a priori estimates.

\indent As it had been remarked in \cite{6}, the question on the solvability of Dirichlet problem for equations \eqref{1} when $B(x,z,p) \not \equiv 0$ is an open one. To investigate this problem, instead of the function $F(\omega ),$ we will consider in the below the following function of  matrix variable,
\begin{equation}\label{4}
F(R)= \log (\det R ),
\end{equation}
where $R =[R_{ij}] \in {\mathbb {R}}^{n \times n},$ which is represented by the form
\begin{equation*}
 R=\omega +\beta , \ \omega ^T=\omega , \omega > 0, \beta ^T= -\beta . 
\end{equation*}
We will show in the below that $\det \beta \geq 0$ and will prove that
$$ \det R = \det (\omega  + \beta ) \geq \det \omega + \det \beta >0.$$
Thus the matrix $R=\omega +\beta$ is always non-singular, the function $F(R)$ is well-defined and infinitely differentiable. The function $F(R)$ is called the {\it Monge-Amp\`{e}re type function,} associated to equation \eqref{1}.

Suppose that $ \delta, \mu $ are fixed nonnegative numbers, where $\delta \in [0,1).$ For the function $ F(R),$ we consider the following set of matrices:
\begin{equation}\label{5}
D_{\delta,\mu} \equiv \left \{ R  \mid  R = \omega + \beta , \omega^T = \omega, \beta^T = -\beta , \lambda_{\min}(\omega)>0,   \delta \lambda_{\min}(\omega) \geq \mu , \| \beta \| \leq \mu  \right \}.
\end{equation}
Here and in what follows, $\|\cdot \|$ denotes the operator norm on $\mathbb{R}^{n \times n}.$ It is easy to verify that $D_{\delta,\mu}$ is an unbounded and convex set in $\mathbb {R}^{n \times n}.$ If $\delta=0$ then $\mu=0,$ $\beta=0$ and the set $D_{0,0}$ consists of symmetric positive definite matrices. In order to generalise the notion of usual concavity for the function $\log (\det \omega ),$ we introduce the so-called $d$-concavity for the function $F(R).$

\begin{definition} \label {Def1}
Suppose that $d \geq 0$ is a nonnegative number. The function $F(R)$ is said to be {\it $d$-concave} in the set $D_{\delta,\mu}$ if for any matrices $R^{(0)}= \left [ R^{(0)}_{ij}\right ]_{n\times n}$ and $ R^{(1)} =  \left [ R^{(1)}_{ij}\right ]_{n\times n}$ from $D_{\delta,\mu},$ we have
\begin{equation*}
 F \bigl ( R^{(1)} \bigl ) - F \bigl ( R^{(0)} \bigl )  \leq  \sum _{i,j=1}^n \frac {\partial F \bigl ( R^{(0)}\bigl )}{\partial R_{ij}} \Bigl ( R^{(1)}_{ij} - R^{(0)}_{ij} \Bigl ) + d.
\end{equation*}
\end{definition}

\indent When $d=0,$ the $0$-concavity is indeed the usual concavity. One of our main results in this paper is the Theorem \ref{Theo3}, in which we prove that the function $F(R)$ is $d$-concave in the set $D_{\delta,\mu}$ with some $d \geq 0,$ which depends only on $\delta $ and $n.$ 

\indent Another aspect of our studying in this paper is the comparison principle for nonsymmetric Monge-Amp\`{e}re type equations \eqref{1}. It is well-known that when $B(x,z,p) \equiv 0,$ the comparison principle holds for elliptic solutions to the equation \eqref{2}. In \cite{2}, this principle had been considered for fully nonlinear second-order elliptic equations. However, in applying to the equation \eqref{1} to compare functions $u(x), v(x) \in C^2 (\overline{\Omega})$, the following condition needs to be satisfied: for any $t \in [0,1],$ the matrix $\omega (x,(1-t)u(x)+tv(x))$ must be positive definite for all $x \in \Omega.$ But, in general, the equation \eqref{1} do not satisfy this condition. The new point of this paper is that we can prove in the Theorem \ref{Theo4} the comparison principle to the equation \eqref{1} in the case when it is $\delta$-elliptic with respect to compared functions.

\begin{definition} \label {Def2} Suppose that $ \delta \in [0,1) $ is a fixed number. We say that the equation \eqref{1} is {\it $\delta-$elliptic with respect to function $u(x) \in C^2 (\Omega)$} if it is elliptic with respect to $u(x)$ and
\begin{equation*} 
 \delta \lambda_{\min} (\omega(x,u)) \geq \mu (B), \  \forall  x \in \Omega,
\end{equation*}
where $\omega(x,u)$ is defined by \eqref{3} and
\begin{equation} \label{6}
 \mu (B) \equiv \sup_{\Omega \times \mathbb{R} \times \mathbb{R}^n} \| B(x,z,p) \| ,
\end{equation}
which is assumed to be finite. 
\end{definition}

\indent Based on the two results mentioned above, in our incoming paper \cite{3}, we will get a priori estimates in $C^{2,\alpha}(\overline{\Omega})$ for $\delta-$elliptic solutions to the Dirichlet problem for \eqref{1}. Moreover, by the method of continuity we will prove in that paper that when $A(x,z,p), f(x,z,p)$ satisfy some conditions which are like those for the Dirichlet problem for \eqref{2} (\cite{4}, \cite{5}, \cite{6}), there exists a unique $\delta$-elliptic solution to the Dirichlet problem for \eqref{1} in the space $C^{2,\alpha } (\overline{\Omega})$ for some $0 < \alpha < 1,$ provided that the matrix $B(x,z,p)$ is sufficiently small.

\indent The structure of the paper is as follows. In Section \S 2, we recall the notion of the $2^{\rm nd}$ compound and its properties for square matrices. In Section \S 3, we will study the second differentials for the Monge-Amp\`{e}re type function $F(R)$ in the set $D_{\delta,\mu}$ and prove its $d$-concavity. In the last section, we will prove the comparison principle for the equations \eqref{1}, which are $\delta-$elliptic with respect to compared functions.

\section{ The $2^{\rm nd}$ compounds of square matrices }

\begin{definition}\rm{({\cite{1}})}\label {Def3}
Let $M=[M_{ij}]$ be an $n \times n$ matrix with entries in $\mathbb{R}$ or $\mathbb{C}$. Suppose that $i < k$ and $j < \ell .$ We denote by $M^{(2)}_{ik, j\ell}$ the minor, which is the determinant at the intersection of rows $i, k$ and columns $j, \ell$ of the matrix $M$, that is,
\begin{equation*}
M^{(2)}_{ik, j\ell} = 
\begin{vmatrix}
M_{ij} & M_{i\ell}\\
M_{kj} & M_{k\ell}
\end{vmatrix}.
\end{equation*}

\indent When the paires $(ik), (j\ell)$ with $i<k$ and $j < \ell $ are arranged in the lexical order, the resulting $\binom{n}{2}\times \binom{n}{2}$ matrix, consisting of corresponding minors is called the $2^{\rm nd}$ {\it compound} of the matrix $M$ and written as $M^{(2)}$. In symbols, we write
\begin{equation*}
M^{(2)}= \left [M^{(2)}_{ik, j\ell}\right ]_{\binom{n}{2}\times \binom{n}{2}}.
\end{equation*}

\end{definition}

\indent Some principal properties of the $2^{\rm nd}$ compound matrices are listed in the following proposition.

\begin{proposition} \rm{({\cite{1}})} \label{Pro1}
Let $M$ and $N$ be matrices in $\mathbb{C}^{n \times n}.$ Then we have the following assertions:\\
\indent {\rm (i)} Binet-Cauchy Theorem: 
\begin{equation*}
 (MN)^{(2)}= M^{(2)}N^{(2)}.
\end{equation*}
\indent {\rm (ii)} $ \bigl ( M^{(2)}\bigl )^T  ={\big (M^T \big )}^{(2)},$ where $M^T$ is the transpose of $M.$\\
\indent {\rm (iii)} $\overline { M^{(2)}}  =\overline {M}^{(2)},$ where $\overline {M}$ is the complex conjugate of $M.$ \\
\indent {\rm (iv)} ${\bigl( M^{(2)}\bigl )}^*  =\big ( M^*\big )^{(2)},$ where $M^*$ is the Hermitian adjoint of $M$, $M^*={\overline{M}}^T$.\\
\indent {\rm (v)} $M$ is non-singular if and only if $M^{(2)}$ is non-singular, and 
$$ \bigl (M^{(2)}\bigl )^{-1}= {\big (M^{-1} \big )}^{(2)}.$$
\indent {\rm (vi)} Suppose that $M \in \mathbb{R}^{n \times n}$ and $M$ is symmetric or skew-symmetric, then $M^{(2)}$ is symmetric.\\
\indent {\rm (vii)}  $ \big (kM\big )^{(2)} = k^2M^{(2)}, \forall k \in \mathbb C. $\\
\indent {\rm (viii)} If $M={\diag} (\lambda_1, \ldots , \lambda_n),$ then
\begin{equation*}
 M^{(2)} = {\diag} (\lambda_j\lambda_k, j<k ).
\end{equation*}
\end{proposition}

To investigate the $d$-concavity of function $F(R)$ in the next section, we need the following proposition. 
\begin{proposition}\label{Pro2}
Let $M=[M_{ij}]$ be a square matrix of order $n.$ Then 
\begin{equation}\label{7}
    M^{(2)}+{\left (M^T\right )}^{(2)} = \frac {1}{2}\left (M+M^T\right )^ {(2)} + \frac {1}{2}\left (M-M^T\right )^ {(2)}.
\end{equation}
\end{proposition}

\begin{proof}
For all $i,j,k,\ell = 1, \ldots, n$ such that $i<k, j<\ell$, we have
\begin{equation*}
\begin{split}
\left (M+M^T \right )^{(2)}_{ik,j\ell}&=
\begin{vmatrix}
M_{ij}+M_{ji} & M_{i\ell}+M_{\ell i}\\
M_{kj}+M_{jk} & M_{k \ell}+M_{\ell k}
\end{vmatrix}
\\
&= 2 \left (
\begin{vmatrix}
M_{ij} & M_{i\ell}\\
M_{kj} & M_{k\ell}
\end{vmatrix}
     +       \begin{vmatrix}
M_{ji} & M_{\ell i}\\
M_{jk} & M_{\ell k}
\end{vmatrix}
            \right )\\
&- \left ( 
  \begin{vmatrix}
M_{ij} & M_{i\ell}\\
M_{kj} & M_{k\ell}
\end{vmatrix}
+
 \begin{vmatrix}
M_{ji} & M_{\ell i}\\
M_{jk} & M_{\ell k}
\end{vmatrix}
-
 \begin{vmatrix}
M_{ij} & M_{\ell i}\\
M_{kj} & M_{\ell k}
\end{vmatrix}
-
 \begin{vmatrix}
M_{ji} & M_{i\ell}\\
M_{jk} & M_{k\ell}
\end{vmatrix}
 \right )\\
&=2 \left (
\begin{vmatrix}
M_{ij} & M_{i\ell}\\
M_{kj} & M_{k\ell}
\end{vmatrix}
     +       \begin{vmatrix}
M_{ji} & M_{\ell i}\\
M_{jk} & M_{\ell k}
\end{vmatrix}
\right )-
\begin{vmatrix}
M_{ij}-M_{ji} & M_{i\ell}-M_{\ell i}\\
M_{kj}-M_{jk} & M_{k\ell}-M_{\ell k}
\end{vmatrix}\\
&=2 \left ( M^{(2)}_{ik,j\ell}+ \left ( M^T \right )^{(2)}_{ik,j\ell} \right )- \left ( M-M^T \right )^{(2)}_{ik,j\ell}.
\end{split}
\end{equation*}
This implies the desired equality \eqref{7}.
\end{proof}

\section{The $d$-concavity of the nonsymmetric Monge-Amp\`{e}re type functions }

\subsection{ Some properties of matrices $R$ belonging the set $D_{\delta,\mu} $  }
Let $D_{\delta, \mu}$ is the set given in \eqref{5}. We shall introduce some properties of matrices $R=\omega+\beta$ from $D_{\delta, \mu}$.  

\begin{proposition}\label {Pro3} 
Suppose that $R=\omega +\beta \in {\mathbb {R}}^{n \times n},$ where $ \omega $ is symmetric positive definite, $\beta $ is skew-symmetric. Then \\
\indent {\rm (i)} $\det R \geq \det \omega  + \det \beta  \geq \det \omega > 0$.\\
\indent {\rm (ii)} Particularly, when $n=2,$
$$ \det R=\det \omega +\det \beta \geq \det \omega   >0.$$
Consequently, $\det R > 0$ and $R$ is always non-singular when $\omega > 0.$ 
\end{proposition}

\begin{proof}
(i) Set
\begin{equation}\label{8}
 \sigma = \omega ^{-\frac {1}{2}} \beta \omega ^{-\frac {1}{2}}.
\end{equation}
Then $\sigma$ is skew-symmetric and
\begin{equation}\label{9}
 R=\omega +\beta = \omega ^{\frac {1}{2}}\left ( E+\omega ^{-\frac {1}{2}} \beta \omega ^{-\frac {1}{2}} \right ) \omega ^{\frac {1}{2}} = \omega ^{\frac {1}{2}}(E+\sigma ) \omega ^{\frac {1}{2}}.
\end{equation}
Set
\begin{equation}\label {10}
D_1 = {\rm diag}\, ( {\im} \sigma_1, \ldots , {\im} \sigma_n ),
\end{equation}
where ${\im} \sigma_1, \ldots , {\im} \sigma_n$ are the eigenvalues of $\sigma,$ ${\im}$ is the imaginary unit, $\sigma_j \in \mathbb{R}, j=1, \ldots, n$ and
\begin{equation}\label{11}
 \sigma_{2j-1} = -\sigma_{2j}, j=1,2, \ldots, \left [ \frac {n}{2} \right ] \ \text{and} \ \sigma_{n}=0  \  \text{if}\  n \ \text{is odd}.
\end{equation}
Then we can write for some unitary matrix $C_1 \in \mathbb{C}^{ n \times n},$
\begin{equation}\label{12}
  \sigma = C_{1}D_{1}C_{1}^{*},
\end{equation}
where $C_1^*$ is the Hermitian adjoint of $C_1$, $C_1^*=C_1^{-1}$. It follows from \eqref{9} and \eqref{12} that
\begin{equation*}
R = \omega ^{\frac {1}{2}}C_{1}(E+D_{1})C_{1}^{*}\omega ^{\frac {1}{2}},
\end{equation*}
which together with \eqref{10} and \eqref{11} yields
\begin{equation} \label{13}
 \det R  =  (\det \omega ) \left ( 1+\sigma_1^2 \right ) \left ( 1+\sigma_3^2 \right )  \cdots \left (1+\sigma_{2 \left [\frac {n}{2}\right ]-1}^2 \right ) .
\end{equation}
Also from \eqref{11}, we have
\begin{equation*}
\det \sigma =0 \ \text{if} \ n \ \text{is odd}, \ \    \det \sigma = \sigma_1^2 \sigma_3^2 \cdots  \sigma_{n-1}^2 \ \text{if} \ n \ \text{is even}. 
\end{equation*}
It follows that
\begin{equation}\label{14}
0 \leq \det \sigma  \leq  \sigma_1^2\sigma_3^2 \cdots  \sigma_{2 \left [\frac {n}{2}\right ]-1}^2.
\end{equation}
This together with \eqref{8} gives
\begin{equation}\label{15}
\det \beta= (\det \omega) (\det \sigma) \geq 0.
\end{equation}
Combining \eqref{13}-\eqref{15}, we obtain 
$$ \det R \geq (\det \omega )( 1+\det \sigma) =\det \omega +\det \beta \geq \det  \omega >0. $$

(ii) When $n=2,$ $\det \sigma=\sigma_1^2$. We infer from this, \eqref{13} and \eqref{15} that
$$ \det R=(\det \omega) (1+\sigma_1^2)= \det \omega + (\det \omega) (\det \sigma)= \det \omega+\det \beta \geq \det \omega>0. $$
The proof is completed.
\end{proof}

\begin{proposition}\label {Pro4} 
Suppose that $R=\omega +\beta \in D_{\delta, \mu}$ and the matrix $\sigma$ is given in \eqref{8}. Then the following assertions hold:\\
\indent {\rm{(i)}} $\| \sigma \| \leq \delta < 1.$\\
\indent {\rm{(ii)}} All eigenvalues ${\im}\sigma_j$ of $\sigma$ satisfy: $|\sigma_j|  \leq \delta <1, j=1,\ldots , n .$
\end{proposition}

\begin{proof}
(i) Since $R=\omega +\beta \in D_{\delta, \mu},$ we have $\delta \lambda_{\min}(\omega) \geq \mu$ and $ \| \beta \| \leq \mu.$ From these estimates and \eqref{8}, we obtain
\begin{equation*}
  \| \sigma \|  \leq \bigl \| \omega ^{-\frac {1}{2}} \bigl \|^2  \| \beta \|  \leq \frac {1}{\lambda_{\min}(\omega)}\mu \leq \delta <1. 
\end{equation*}

(ii) The estimate (ii) follows directly from (i) and the fact that $|\sigma_j | \leq \| \sigma \| , j=1, \ldots ,n.$ 
\end{proof}

\begin{proposition}\label{Pro5}
Suppose that $R=\omega +\beta \in D_{\delta, \mu}$. Then
\begin{equation}\label{16}
\frac {1}{\delta ^n}\|\beta \|^n + \bigl (2 ^{[\frac {n}{2}]}-1\bigl ) \det \beta \leq \det \omega+  \bigl (2 ^{[\frac {n}{2}]}-1\bigl ) \det \beta      \leq \det R \leq (1+\delta^2) ^{[\frac {n}{2}]} \det \omega,
\end{equation}
where, when $\delta=0$ we have $\beta=0$ and $\dfrac{0}{0}=0.$
\end{proposition}

\begin{proof}
By \eqref{13},
$$ \det R  =  (\det \omega ) \left ( 1+\sigma_1^2 \right ) \left ( 1+\sigma_3^2 \right )  \cdots \left (1+\sigma_{2 \left [\frac {n}{2}\right ]-1}^2 \right ) .$$
By Proposition \ref{Pro4}, $|\sigma_j|  \leq \delta <1, j=1,\ldots , n.$ Thus
\begin{equation*}
\begin{split}
 (1+\delta^2 )^{ [\frac {n}{2}]}& \geq  \left ( 1+\sigma_1^2 \right ) \left ( 1+\sigma_3^2 \right )     \cdots \Bigl (1+\sigma_{2 \left [\frac {n}{2}\right ]-1}^2 \Bigl )\\
&\geq 1 + \bigl (2^{ [\frac {n}{2}]}-1 \bigl ) \sigma_1^2\sigma_3^2 \cdots  \sigma_{2 \left [\frac {n}{2}\right ]-1}^2 \geq 1+ \bigl (2^{ [\frac {n}{2}]}-1 \bigl )  \det \sigma ,
\end{split}
\end{equation*}
where the last inequality is by \eqref{14}. Moreover, we have
\begin{equation*}
\det \omega \geq \left (\lambda_{\min}(\omega)\right )^n \geq \frac {1}{\delta ^n} \mu ^n \geq \frac {1}{\delta ^n} \|\beta \|^n .
\end{equation*}
From these estimates and \eqref{15}, we obtain the conclusion of Proposition \ref{Pro5}.
\end{proof}

\begin{proposition}\label {Pro6}
Suppose that $R=\omega +\beta \in D_{\delta, \mu}$ and the matrix $\sigma $ is given in \eqref{8}. Then
\begin{equation}\label{17}
\begin{split}
\frac {{ R^{-1}+  (R^{-1})}^T }{2} &= \omega ^{-\frac {1}{2}} \left ( E - \sigma^2  \right )^{-1}   \omega ^{-\frac {1}{2}},  \\
\frac { { R^{-1}- (R^{-1})}^T }{2} &= \omega ^{-\frac {1}{2}}(-\sigma) \left ( E - \sigma^2  \right )^{-1}   \omega ^{-\frac {1}{2}}.
\end{split}
\end{equation}
\end{proposition}

\begin{proof}
It follows from \eqref{9} that
$$R^{-1} = \omega ^{-\frac {1}{2}} \left ( E + \sigma  \right )^{-1}   \omega ^{-\frac {1}{2}},$$
 $$\big ( R^{-1}\big )^T= \omega ^{-\frac {1}{2}} \big ( (E + \sigma )^{-1}\big )^T \omega ^{-\frac {1}{2}}  =\omega ^{-\frac {1}{2}} \left ( E - \sigma  \right )^{-1}   \omega ^{-\frac {1}{2}}  .$$
Thus
\begin{equation}\label{18}
\begin{split}
\frac {R^{-1}+{ \left ( R^{-1}\right )}^T}{2} &= \omega ^{-\frac {1}{2}} \frac {\left ( E + \sigma  \right )^{-1}+\left ( E - \sigma  \right )^{-1}}{2}  \omega ^{-\frac {1}{2}},  \\
\frac {R^{-1}-{ \left ( R^{-1}\right )}^T}{2} &= \omega ^{-\frac {1}{2}} \frac {\left ( E + \sigma  \right )^{-1}-\left ( E - \sigma  \right )^{-1}}{2}  \omega ^{-\frac {1}{2}}. 
\end{split}
\end{equation}
Note that $E-\sigma^2= (E-\sigma)(E+\sigma),$ so we have
\begin{align*}
\frac {\left ( E + \sigma  \right )^{-1}+\left ( E - \sigma  \right )^{-1}}{2} \big ( E-\sigma^2 \big )&=\frac{(E-\sigma) + (E+\sigma)}{2}=E,\\
\frac {\left ( E + \sigma  \right )^{-1}-\left ( E - \sigma  \right )^{-1}}{2} \big ( E-\sigma^2 \big )&=\frac{(E-\sigma) - (E+\sigma)}{2}=-\sigma.
\end{align*}
Therefore,
\begin{equation*}
\begin{split}
\frac {\left ( E + \sigma  \right )^{-1}+\left ( E - \sigma  \right )^{-1}}{2} &= \left ( E - \sigma ^2 \right )^{-1},  \\
\frac {\left ( E + \sigma  \right )^{-1}-\left ( E - \sigma  \right )^{-1}}{2} &= (-\sigma) \left ( E - \sigma ^2 \right )^{-1}. 
\end{split}
\end{equation*}
From these equalities and \eqref{18}, we obtain the desired equalities in \eqref{17}.
\end{proof}

\begin{corollary}\label {Cor1}
Suppose that $R=\omega +\beta \in D_{\delta, \mu}$ and suppose that the matrix $\sigma = \omega ^{-\frac {1}{2}} \beta \omega ^{-\frac {1}{2}}$ is diagonalised by a unitary matrix $C_{1} \in \mathbb{C}^{n \times n}$ as in \eqref{12},
\begin{equation*}
  \sigma = C_{1}D_{1}C_{1}^{*},
\end{equation*}
where $D_1$ is the diagonal matrix given by \eqref{10}.

Then 
\begin{equation}\label{19}
\begin{split}
\frac {R^{-1}+{ \left ( R^{-1}\right )}^T}{2} &= \omega ^{-\frac {1}{2}} C_1 D_2 C_1^{*}  \omega ^{-\frac {1}{2}},  \\
\frac {R^{-1}-{ \left ( R^{-1}\right )}^T}{2} &= \omega ^{-\frac {1}{2}} C_1 D_3 C_1^{*}  \omega ^{-\frac {1}{2}},
\end{split}
\end{equation}
where
\begin{equation}\label{20}
\begin{split}
D_{2} &= \left (E-D_{1}^2 \right )^{-1} = {\diag}\left (\frac {1}{1+\sigma_1^2}, \ldots , \frac {1}{1+\sigma_n^2} \right ),\\
D_{3} &= (-D_{1})\left (E-D_{1}^2 \right )^{-1} = {\diag} \left (\frac {-{\im}\sigma_1}{1+\sigma_1^2},  \ldots , \frac {- {\im}\sigma_n}{1+\sigma_n^2} \right ).
\end{split}
\end{equation}
\end{corollary}

\begin{proof}
All equalities in \eqref{19}, \eqref{20} are followed easily from \eqref{10}, \eqref{12} and \eqref{17}.
\end{proof}

\begin{corollary}\label{Cor2}
Suppose that $R=\omega+\beta \in D_{\delta, \mu}$. Then
\begin{equation}\label{21}
  \frac{1}{1+\delta ^2}{\Tr}\,\omega^{-1}  \leq {\Tr} R^{-1} \leq  {\Tr} \,\omega^{-1},
\end{equation}
here and in the below, ${\Tr}$ stands for the trace operator of square matrices.
\end{corollary}

\begin{proof}
From \eqref{19} and \eqref{20}, we have
\begin{equation}\label{22}
\begin{split}
{\Tr} R^{-1}&= {\Tr} \left (\frac{R^{-1} + \left (R^{-1} \right )^T}{2}    \right )={\Tr}\, \Bigl (\omega ^{-\frac{1}{2}} C_1 D_2 C_1^{*} \omega ^{-\frac{1}{2}} \Bigl )\\
&={\Tr}\, \big ( D_2 C_1^{*} \omega ^{-1}C_1 \big ) =\sum_{j} {(D_2)}_{jj} \big ( C_1^{*} \omega ^{-1} C_1 \big )_{jj}.
\end{split}
\end{equation}
Note that $\omega^{-1}$ is positive definite, $C_1$ is unitary and, by Proposition \ref{Pro4}, $\dfrac{1}{1+\delta^2} \leq {(D_2)}_{jj} \leq 1, j=1, \ldots,n.$ We then obtain from \eqref{22} that
$$ \frac{1}{1+\delta^2}{\Tr}\, \omega^{-1}\leq  \frac{1}{1+\delta^2} \sum_{j}  \big ( C_1^{*} \omega ^{-1} C_1 \big )_{jj}    \leq     { \Tr} R^{-1} \leq \sum_{j}  \big ( C_1^{*} \omega ^{-1} C_1 \big )_{jj}= {\Tr}\, \omega^{-1}.  $$
This completes the proof.
 \end{proof}

\subsection{The second differentials of the nonsymmetric Monge-Amp\`{e}re type functions}
\begin{proposition} \label{Pro7}
Let $F(R)$ be the function given by \eqref{4}, where $\det R>0.$ Let $R^{-1}= \left [ R^{ij} \right ]$ denote the inverse of $R=[R_{ij}].$ Then for all $i,j,k,\ell=1, \ldots,n,$ we have that
\begin{align}
 F^{ij}&\equiv \frac{\partial F (R)}{ \partial R_{ij}}= R^{ji},  \label{23}\\
 F^ {ij,k\ell}& \equiv \frac{\partial^2 F (R)}{\partial R_{ij} \partial R_{k \ell}}= -R^{\ell i}R^{jk}.\label{24}
\end{align}
\end{proposition}

\begin{proof}
Let $U=[U_{ij}]$ denote the cofactor matrix of $R,$ i.e., $U^T= (\det R)R^{-1}.$ For a fixed $i, i=1, \ldots, n,$ we expand the determinant $\det R$ according to the $i-$th row,
$$ \det R= R_{i1} U_{i1}+\cdots + R_{in}U_{in}. $$
Then 
$$ \frac{\partial F (R)}{ \partial R_{ij}}= \frac{1}{\det R} \frac{\partial  (\det R) }{ \partial R_{ij}}= \frac{1}{\det R} U_{ij}=R^{ji},\ \text{for}\ i,j=1, \ldots, n.  $$
Thus \eqref{23} is proved.

It follows from \eqref{23} that
\begin{equation*}
 \sum_{p} R_{sp} F^{ip} =\sum_{p} R_{sp} R^{pi}= \delta_{is}, \  \text{for}\  i,s=1,\ldots, n.
\end{equation*}
Differentiating this equation with respect to $R_{k\ell},$ we get
$$ \sum_{p} \frac{ \partial R_{sp } }{ \partial R_{k\ell}} F^{ip} +\sum_{p} R_{sp} F^{ip,k\ell}=0, $$
and thus
\begin{equation*}
   \delta _{sk} F^{i\ell} +\sum_{p} R_{sp} F^{ip,k\ell}=0, \ \text{for} \ i,s,k,\ell=1, \ldots, n.
\end{equation*}
Multiplying this equality by $R^{js}$ and summing over $s,$ we have for $i,j,k,\ell=1, \ldots, n,$
$$  \sum_{s} \delta _{sk}F^{i\ell} R^{js}   + \sum_{p,s} R^{js} R_{sp} F^{ip,k\ell}=0,  $$
or
$$R^{\ell i} R^{jk}  +F^{ij,k\ell}=0 ,$$
which gives the required result \eqref{24}. The proof is now completed.
\end{proof}

\vspace{10pt}
Now we consider the second order differentials of the function $F(R)$ given by \eqref{4}, where $R \in D_{\delta, \mu}$, $D_{\delta, \mu}$ is the unbounded and convex set given in \eqref{5}. Let $M=[M_{ij}] \in \mathbb{R}^{n \times n}$. We consider the function $\mathcal{F}$ defined as follows:
\begin{equation}\label{25}
\begin{split}
\mathcal{F}(R,M)&:  D_{\delta, \mu} \times \mathbb{R}^{n \times n} \rightarrow \mathbb{R},\\
\mathcal{F}(R,M)&=\sum_{i,j,k,\ell}\frac{\partial^2 F}{\partial R_{ij} \partial R_{k\ell}} M_{ij}M_{k\ell }\\
& = - \sum_{i,j,k,\ell}R^{\ell i}R^{jk} M_{ij}M_{k\ell }.
\end{split}
\end{equation}

\begin{proposition}\label {Pro8}
Suppose $R\in D_{\delta, \mu}.$ Then for any matrix $M=P+Q \in \mathbb{R}^{n \times n},$ the following equality holds
\begin{equation}\label{26}
\mathcal{F}(R,M)=\mathcal{F}(R,P) + \mathcal{F}(R,Q) +2 \mathcal{L}(R,P,Q),
\end{equation}
where
\begin{equation}\label{27}
\mathcal{L}(R,P,Q) = - \sum_{i,j,k,\ell}R^{\ell i}R^{jk} P_{ij}Q_{k\ell }.
\end{equation}
\end{proposition}

\begin{proof}
It follows from \eqref{25} that
\begin{equation*}
\begin{split}
\mathcal{F}( R,M)& = - \sum_{ i,j,k,\ell} R^{\ell i} R^{jk}  P_{ij} P_{k\ell} - \sum_{ i,j,k,\ell} R^{\ell i} R^{jk}  Q_{ij} Q_{k\ell} \\
& - \sum_{ i,j,k,\ell} R^{\ell i} R^{jk}  P_{ij} Q_{k\ell} - \sum_{ i,j,k,\ell} R^{\ell i} R^{jk} P_{k\ell} Q_{ij} .
\end{split}
\end{equation*}
Note that
\begin{equation*}
 \sum_{ i,j,k,\ell} R^{\ell i} R^{jk} P_{k\ell}  Q_{ij}  \overset {   \substack {\ell \leftrightarrow j \\ i \leftrightarrow k }}{=}  \sum_{ i,j,k,\ell} R^{jk} R^{\ell i}P_{i j}  Q_{k\ell} . 
\end{equation*}
Combining these equalities, we obtain the conclusion of Proposition \ref{Pro8}.
\end{proof}

\begin{proposition}\label {Pro9}
Suppose that $R \in D_{\delta, \mu}.$ Then for any symmetric matrix $P \in \mathbb{R}^{n \times n},$ the following equality holds
\begin{equation}\label{28}
\mathcal{F}(R,P)=- \left [ \mathcal{G}(R,P) \right ]^2 +\mathcal{H}(R,P),
\end{equation}
where
\begin{equation}\label{29}
\begin{split}
\mathcal{G}(R,P)&={\Tr} \big ( R^{-1}P  \big ),\\
\mathcal{H}(R,P)&= 2 \,{\Tr} \left [ \left ( R^{-1} \right )^{(2)}  P^{(2)}\right ].
\end{split}
\end{equation}
\end{proposition}

\begin{proof}
From \eqref{25} and the fact that $P^T=P,$ we have
\begin{equation*}
\begin{split}
\mathcal{F}(R,P) &= -\sum_{i,j,k,\ell} R^{\ell i} R^{jk} P_{ij} P_{k\ell} = -\sum_{i,j,k,\ell} R^{\ell i} R^{jk} P_{ji} P_{\ell k}\\
& \overset {   \substack {i \leftrightarrow \ell \\ j \leftrightarrow k }}{=} - \sum_{i,j,k,\ell} R^{i\ell} R^{kj} P_{k\ell} P_{i j}
\overset { j \leftrightarrow \ell}{=} -\sum_{i,j,k,\ell} R^{i j} R^{k\ell} P_{kj} P_{i \ell}. 
\end{split}
\end{equation*}
It follows that
\begin{equation*}
\begin{split}
\mathcal{F}(R,P)&=-\frac{1}{2} \Biggl ( \sum_{i,j} R^{ij} P_{ij} \Biggl ) \Biggl ( \sum_{k,\ell} R^{k\ell} P_{k\ell} \Biggl ) -\frac{1}{2} \Biggl ( \sum_{i,\ell} R^{i\ell} P_{i\ell} \Biggl ) \Biggl ( \sum_{j,k} R^{kj} P_{kj} \Biggl ) \\
&+\frac{1}{2} \Biggl [ \Biggl (  \sum_{i,j} R^{ij} P_{ij} \Biggl )   \Biggl (  \sum_{k,\ell} R^{k\ell} P_{k\ell} \Biggl )   + \Biggl (  \sum_{i,\ell} R^{i\ell} P_{i\ell}  \Biggl )  \Biggl (  \sum_{j,k} R^{kj} P_{kj} \Biggl ) \\
&\quad\quad\quad\quad\quad\quad\quad\quad \quad\quad - \sum_{i,j,k,\ell} R^{i\ell} R^{kj} P_{k \ell} P_{ij} -  \sum_{i,j,k,\ell} R^{ij} R^{k \ell} P_{kj} P_{i\ell} \Biggl ] \\
&=- \Biggl ( \sum_{i,j} R^{ij} P_{ij} \Biggl ) ^2+2\sum_{\substack{i<k \\ j<\ell}  } \big ( R^{ij} R^{k\ell}-R^{i\ell} R^{kj} \big )   (P_{ij} P_{k\ell} - P_{i\ell} P_{kj} )\\
&=- \Biggl ( \sum_{i,j} R^{ij} P_{ij} \Biggl ) ^2+ 2 \sum_{\substack{i<k \\ j<\ell}} \Bigl ( \big ( R^{-1}\big)^{(2)}_{ik,j\ell } P^{(2)}_{j\ell ,ik}  \Bigl )\\
&=- \left [ {\rm Tr} \left ( R^{-1}P  \right ) \right ]^2+ 2 \, {\rm Tr} \left [ \left ( R^{-1} \right )^{(2)}  P^{(2)}\right ].
\end{split}
\end{equation*}
This completes the proof.
\end{proof}

\begin{proposition}\label {Pro10}
Suppose that $R \in D_{\delta, \mu}$. Then for any skew-symmetric matrix $Q \in \mathbb{R}^{n \times n},$ the following equality holds
\begin{equation}\label{30}
\mathcal{F}(R,Q)=- \left [ \mathcal{G}(R,Q)\right ]^2 +\mathcal{H}(R,Q),
\end{equation}
where the functions $\mathcal{G}$ and $\mathcal{H}$ are defined as in \eqref{29}.
\end{proposition}

\begin{proof}
The proof is similar to that of Proposition \ref{Pro9}.
\end{proof}

\begin{proposition}\label {Pro11}
Suppose that $R \in D_{\delta, \mu}.$ Then for any symmetric matrix $P\in \mathbb{R}^{n \times n}$ and any skew-symmetric matrix $Q \in \mathbb{R}^{n \times n}$, the following equality holds
\begin{equation}\label{31}
\mathcal{L}(R,P,Q)=-\frac{1}{2}\, {\Tr} \left [ \left ( R^{-1}- {(R^{-1})}^T \right )  P \left ( R^{-1}+{(R^{-1})}^T \right ) Q  \right ],
\end{equation}
where $\mathcal{L}(R,P,Q)$ is defined by \eqref{27}.
\end{proposition}

\begin{proof}
Note that
\begin{equation*}
\begin{split}
\mathcal{L}(R,P,Q)&= - \sum_{i,j,k,\ell} R^{\ell i} R^{jk} P_{ij} Q_{k\ell} \overset {   \substack {i \leftrightarrow j \\ k \leftrightarrow \ell }}{=} - \sum_{i,j,k,\ell} R^{kj} R^{i\ell} P_{ji} Q_{\ell k}.
\end{split}
\end{equation*}
From this and the fact that $P^T=P, Q^T=-Q,$ we get
\begin{equation*}
\begin{split}
&\mathcal{L}(R,P,Q)= - \frac{1}{2} \sum_{i,j,k,\ell} R^{\ell i} R^{jk} P_{ij} Q_{k\ell}  -\frac{1}{2} \sum_{i,j,k,\ell} R^{kj} R^{i\ell} P_{ji} Q_{\ell k}\\
&= \frac{1}{2} \sum_{i,j,k,\ell} R^{jk} P_{ji}R^{\ell i}Q_{\ell k}    -\frac{1}{2} \sum_{i,j,k,\ell } R^{kj}P_{ji}R^{i \ell} Q_{\ell k}\\
&= \frac{1}{2} \, {\Tr} \left [ {(R^{-1})}^T P{(R^{-1})}^T Q  \right ]    -\frac{1}{2} \, {\rm Tr} \left [ R^{-1} P R^{-1}Q \right ]\\
&=-\frac{1}{2} \, {\Tr} \left [ \left ( R^{-1} -{(R^{-1})}^T  \right )  P{(R^{-1})}^T Q  \right ]  - \frac{1}{2} \, {\rm Tr}  \left [ \left ( { R^{-1}-  (R^{-1})}^T \right )  P R^{-1} Q  \right ]\\
&= -\frac{1}{2}\, {\Tr} \left [ \left ( R^{-1}- {(R^{-1})}^T \right )  P \left ( {(R^{-1})}^T +R^{-1} \right ) Q  \right ],
\end{split}
\end{equation*}
where in the fourth step, we have used the equality 
$${\Tr} \left [ R^{-1} P {(R^{-1})}^T  Q  \right ] =  {\Tr} \left [{(R^{-1})}^T P R^{-1}  Q  \right ] =0,$$
which holds due to the skew-symmetry of $Q$ and the symmetry of matrices \\
$R^{-1} P {(R^{-1})}^T$, ${(R^{-1})}^T P R^{-1}.$ The proof is completed.
\end{proof}

\par For $R=\omega + \beta \in D_{\delta, \mu}$ fixed and for matrix $M \in {\mathbb {R}}^{n \times n},$ we set 
\begin{equation}\label{32}
\tilde{M} \equiv \omega ^{-\frac {1}{2}}M \omega ^{-\frac {1}{2}} = \big [ \tilde{M}_{jk} \big ] ,\ \ \ 
\tilde{\tilde{M}} \equiv C_1^{*}\tilde{M}C_1 = \bigl [ \tilde{\tilde{M}}_{jk} \bigl ],
\end{equation}
where $C_1$ is the unitary matrix defined in \eqref{12}. It is obvious that 
\begin{equation}\label{33}
 \big |\tilde{M} \big | = \bigl | \tilde{\tilde{M}}\bigl |, \ \ \ \big \| \tilde{M} \big \| = \Bigl \| \tilde{\tilde{M}}\Bigl \|,
\end{equation}
where $|\cdot |$ and $ \| \cdot \|$ denote, respectively, the Frobenius norm and the operator norm on $\mathbb{C}^{n \times n},$ which are defined as follows: for any matrix $K=[K_{ij}] \in \mathbb{C}^{n \times n},$
$$  |K| = \left ( \sum_{i,j} |K_{ij}|^2 \right )^{1/2},  \ \ \ \|K\|= \underset{\xi \in \mathbb{C}^n, |\xi|=1}{\sup}\, |K \xi|. $$

\begin{proposition} \label{Pro12}
For any matrix $M \in \mathbb{R}^{n \times n}$, we have the following estimate
\begin{equation}\label{34}
 \big (\lambda_{\max}(\omega)\big )^{-2} |M|^2 \leq  \big | \tilde{M}\big |^2 \leq \big ( \lambda_{\min} (\omega) \big )^{-2} |M|^2,
\end{equation}
where $\lambda_{\max}(\omega) $ and $ \lambda_{\min}(\omega)$ denote, respectively, the largest and smallest eigenvalues of $\omega .$
\end{proposition}

\begin{proof}
Let $\lambda_1, \ldots, \lambda_n$ be the eigenvalues of $\omega,$ where $ \lambda_1 \geq \cdots \geq  \lambda_n >0.$ Write
$$ \omega=CDC^{-1}, $$
where $C$ is orthogonal and $D={\diag} ( \lambda_1, \ldots, \lambda_n ).$ Then
$$ \omega^{-\frac{1}{2}}= CD^{-\frac{1}{2}} C^{-1},\,\,\, D^{-\frac{1}{2}}= {\diag}\Bigl ( \lambda_1 ^{-\frac{1}{2}}, \ldots, \lambda_n ^{-\frac{1}{2}}\Bigl ).  $$
Therefore,
\begin{equation*}
\begin{split}
\big | \tilde{M}\big|^2&=  \Bigl | \omega^{-\frac{1}{2}}M \omega^{-\frac{1}{2}}  \Bigl |^2 = \Bigl | \Bigl ( CD^{-\frac{1}{2}}C^{-1}\Bigl ) M \Bigl ( CD^{-\frac{1}{2}}C^{-1}\Bigl ) \Bigl |^2\\
&=\Bigl | D^{-\frac{1}{2}} \big ( C^{-1}MC \big ) D^{-\frac{1}{2}} \Bigl |^2=\sum_{i,j}  \lambda_i^{-1} \lambda_j^{-1} \bigl (\big ( C^{-1} MC \big )_{ij}\bigl )^2.
\end{split}
\end{equation*}
From this and the fact that $0< \lambda_1^{-1} \leq \lambda_i^{-1} \leq \lambda_n ^{-1} (i=1, \ldots, n),$ we obtain
\begin{equation*}
   \lambda_1^{-2}|M|^2=    \lambda_1^{-2} \sum_{i,j} \bigl (\big (C^{-1}M C \big )_{ij}\bigl )^2 \leq     \big | \tilde{M }\big |^2 \leq \lambda_n^{-2} \sum_{i,j} \bigl (\big (C^{-1}M C \big )_{ij}\bigl )^2 =\lambda_n ^{-2} |M|^2.
\end{equation*}
The proof is completed.
\end{proof}

\begin{proposition}\label {Pro13}
Suppose that $R=\omega +\beta \in D_{\delta, \mu}.$ Then for any symmetric matrix $P\in \mathbb{R}^{n \times n},$ we have
\begin{equation}\label{35}
\mathcal{F}(R,P) = -\sum _{j,k=1}^n \frac {1-  \sigma_j \sigma_k   }{\left ( 1+\sigma_j^2 \right ) \left ( 1+\sigma_k^2 \right )} \left |{\tilde{\tilde P}}_{jk}\right |^2 ,
\end{equation}
where ${\im}\sigma _1, \ldots, {\im}\sigma _n$ are the eigenvalues of the matrix $\sigma$, defined by \eqref{8}.
\end{proposition}

\begin{proof}
Since $P$ is symmetric, by Proposition \ref{Pro1}, $P^{(2)}$ is also symmetric. Hence from \eqref{7}, \eqref{28} and \eqref{29}, we have
\begin{equation} \label{36}
\begin{split}
&\mathcal{F}(R,P)= - \left [ \mathcal{G}(R,P)\right ]^2 +\mathcal{H}(R,P) =- \left [ \mathcal{G}(R,P)\right ]^2 +2\, {\Tr} \left [ \left ( R^{-1} \right )^{(2)}  P^{(2)}\right ] \\
&= -\left [ \mathcal{G}(R,P)\right ]^2 +2\,{\Tr} \left [      \frac {\left ( R^{-1} \right )^{(2) } + \left (\left ( R^{-1} \right )^{T}\right )^{(2)}  }  {2}   P^{(2)}\right ] =-\left [ \mathcal{G}(R,P)\right ]^2 \\
&+2\,{\Tr} \left [ \left ( \frac{ R^{-1}+ \left ( R^{-1}  \right )^T}{2} \right )^{(2)} P^{(2)}  \right ] +2\, {\Tr} \left [  \left ( \frac{ R^{-1} - \left ( R^{-1}  \right )^T}{2} \right )^{(2)} P^{(2)} \right ].
\end{split}
\end{equation}

\noindent It follows from \eqref{19} and \eqref{20} that
\begin{equation} \label{37}
\begin{split}
 \mathcal{G}(R,P) &= {\Tr} \left ( \frac{ R^{-1} +\left ( R^{-1} \right ) ^T }{2}P \right ) ={\Tr} \left ( \omega^{-\frac{1}{2}} C_1 D_2 C_1^*\omega^{-\frac{1}{2}} P  \right ) \\
&={\Tr} \left ( D_2 C_1^*  \omega^{-\frac{1}{2}} P  \omega^{-\frac{1}{2}} C_1 \right ) = {\Tr} \left ( D_2 {\tilde{\tilde P}}\right ) =\sum_{j}\frac{{\tilde{\tilde P}}_{jj} }{1+\sigma_j^2}.
\end{split}
\end{equation}

\noindent Since $P$ is symmetric, $\tilde{\tilde{P}}$ is Hermitian. Hence $\tilde{\tilde{P}}_{jk} \tilde{\tilde{P}}_{kj}= \left | \tilde{\tilde{P}}_{jk} \right |^2,$ $j,k=1, \ldots,n.$ From these equalities, \eqref{19}, \eqref{20} and Proposition \ref{Pro1}, we obtain 
\begin{equation} \label{38}
\begin{split}
&2 \,{\Tr} \left [  \left ( \frac{ R^{-1} + \left ( R^{-1}  \right )^T}{2}   \right )^{(2)} P^{(2)} \right ]=  2 \,{\Tr} \left [  \bigl (  \omega^{-\frac{1}{2}} C_1 D_2 C_1^* \omega^{-\frac{1}{2}} \bigl )^{(2)} P^{(2)} \right ]\\
&= 2 \,{\Tr} \left [ D_2^{(2)}  \left ( C_1^* \right )^{(2)} \bigl ( \omega ^{-\frac{1}{2}}\bigl )^{(2)} P^{(2)} \bigl ( \omega ^{-\frac{1}{2}}\bigl )^{(2)} C_1^{(2)}   \right ]\\
& = 2 \,{\Tr} \left [ D_2 ^{(2)} { \tilde{\tilde{P}}}^{(2)}   \right ] = 2 \sum_{j<k}  \bigl ( D_2 ^{(2)}\bigl )_{jk,jk}{ \tilde{\tilde{P}}}^{(2)}_{jk,jk} \\
&= 2 \sum_{j<k} \frac{\tilde{\tilde{P}}_{jj} \tilde{\tilde{P}}_{kk}- \left | \tilde{\tilde{P}}_{jk} \right |^2 }{\left (1+\sigma_j ^2 \right )  \left (1+\sigma_k ^2 \right )}=\sum_{j\ne k} \frac{\tilde{\tilde{P}}_{jj} \tilde{\tilde{P}}_{kk}}{\left (1+\sigma_j ^2 \right )  \left (1+\sigma_k ^2 \right )}-\sum_{j \ne k}\frac{\left | \tilde{\tilde{P}}_{jk} \right |^2 }{\left (1+\sigma_j ^2 \right )  \left (1+\sigma_k ^2 \right )}\\
&= \sum_{j, k} \frac{\tilde{\tilde{P}}_{jj} \tilde{\tilde{P}}_{kk}}{\left (1+\sigma_j ^2 \right )  \left (1+\sigma_k ^2 \right )}-\sum_{j , k}\frac{\left | \tilde{\tilde{P}}_{jk} \right |^2 }{\left (1+\sigma_j ^2 \right )  \left (1+\sigma_k ^2 \right )}\\
&  = \Biggl (  \sum_{j} \frac{\tilde{\tilde{P}} _{jj}} { 1+ \sigma_j ^2 }   \Biggl )^2 - \sum_{j,k} \frac{ \left | \tilde{\tilde{P}}_{jk} \right |^2 }{\left (1+\sigma_j ^2 \right )  \left (1+\sigma_k ^2 \right ) }.
\end{split}
\end{equation}

\noindent Combining \eqref{37}, \eqref{38} yields
\begin{equation} \label{39}
2\,{\Tr} \left [ \left ( \frac{ R^{-1} + \left ( R^{-1}  \right )^T}{2}   \right )^{(2)} P^{(2)}  \right ] = \left [ \mathcal{G}(R,P) \right ]^2 - \sum_{j,k} \frac{ \left | \tilde{\tilde{P}}_{jk} \right |^2 }{\left (1+\sigma_j ^2 \right )  \left (1+\sigma_k ^2 \right ) }.
\end{equation}

\noindent From \eqref{19}, \eqref{20} and Proposition \ref{Pro1}, we also get
\begin{equation} \label{40}
\begin{split}
&2\,{\Tr} \left [ \left (  \frac{ R^{-1} - \left ( R^{-1}  \right )^T}{2}   \right )^{(2)} P^{(2)}  \right ] = 2\, {\Tr} \left [ \left (  \omega^{-\frac{1}{2} } C_1 D_3 C_1^* \omega^{-\frac{1}{2} }   \right )^{(2 )}  P^{(2)}\right ]\\
&=2 \,{\Tr} \left [ D_3 ^{ (2) } {\tilde{\tilde{P}}}^{(2)}   \right ] =2 \sum_{j<k}  \bigl ( D_3 ^{(2)}\bigl )_{jk,jk}{ \tilde{\tilde{P}}}^{(2)}_{jk,jk}   \\
&= 2 \sum_{j<k} \frac{ -\sigma_j  \sigma_k}{\left (1+\sigma_j ^2 \right )  \left (1+\sigma_k ^2 \right ) } \left (\tilde{\tilde{P}}_{jj} \tilde{\tilde{P}}_{kk}- \left | \tilde{\tilde{P}}_{jk}   \right |^2 \right )\\
&=-\sum_{j \ne k} \frac{ \sigma_j  \sigma_k}{\left (1+\sigma_j ^2 \right )  \left (1+\sigma_k ^2 \right ) } \tilde{\tilde{P}}_{jj}\tilde{ \tilde{P}}_{kk}+ \sum_{j \ne k} \frac{ \sigma_j  \sigma_k}{\left (1+\sigma_j ^2 \right )  \left (1+\sigma_k ^2 \right ) } \left | \tilde{\tilde{P}}_{jk}\right |^2.
\end{split}
\end{equation}

\noindent Obviously, $ {\Tr} \left ( \dfrac{ R^{-1} - \left ( R^{-1} \right ) ^T }{2} P    \right )=0.$ It then follows from \eqref{19} and \eqref{20} that
$$ {\Tr} \bigl ( \omega^{-\frac{1}{2} }C_1 D_3 C_1^* \omega^{-\frac{1}{2}}P \bigl )={\Tr} \left ( D_3 \tilde{\tilde{P}} \right )=\sum_{j} \frac{ -{\im} \sigma_j}{1+\sigma_j^2 } \tilde{\tilde{P}}_{jj}=0, $$
or equivalently,
$$ \Biggl ( \sum_{j} \frac{\sigma_j}{1+\sigma_j^2 } \tilde{\tilde{P}}_{jj} \Biggl ) ^2=0.$$
Hence 
\begin{equation*} 
\sum_{j} \frac{\sigma_j^2 }{ \left ( 1+\sigma_j^2  \right )^2} \left ( \tilde{\tilde{P}}_{jj}\right )^2=- \sum_{j \ne k}\frac{ \sigma_j  \sigma_k}{\left (1+\sigma_j ^2 \right )  \left (1+\sigma_k ^2 \right ) } \tilde{\tilde{P}}_{jj} \tilde{\tilde{P}}_{kk}. 
\end{equation*}
This together with \eqref{40} gives
\begin{equation} \label{41}
2\,{\Tr} \Biggl [ \Biggl (  \frac{ R^{-1} - \big ( R^{-1}  \big )^T}{2}   \Biggl )^{(2)} P^{(2)}  \Biggl ]=\sum_{j ,k} \frac{ \sigma_j  \sigma_k}{\big (1+\sigma_j ^2 \big )  \big (1+\sigma_k ^2 \big ) } \Bigl | \tilde{\tilde{P}}_{jk}\Bigl |^2.
\end{equation}
The proof is straightforward from \eqref{36}, \eqref{39} and \eqref{41}.
\end{proof}

\begin{corollary} \label{Cor3}
Suppose that $R=\omega +\beta \in D_{\delta, \mu}.$ Then for any symmetric matrix $P \in \mathbb{R}^{n \times n},$ we have
\begin{equation}\label{42}
\mathcal{F}(R,P) \leq - \frac { 1-\delta ^2 }{ \left ( 1+\delta ^2 \right )^2} \big |\tilde{P} \big |^2 \leq  - \frac { 1-\delta ^2 }{ \left ( 1+\delta ^2 \right )^2} \big ( \lambda_{\max} (\omega) \big )^{-2} |P|^2 .
\end{equation}
\end{corollary}

\begin{proof}
By Proposition \ref{Pro4}, we have $|\sigma _j| \leq \delta <1,j=1, \ldots, n.$ Hence from \eqref{33} and \eqref{35}, we obtain
\begin{equation*}
\begin{split}
\mathcal{F}(R,P) &= -\sum_{j,k} \frac{1-\sigma_j \sigma_k } { \left (1+\sigma_j^2 \right ) \left (1+\sigma_k^2 \right ) }\left | \tilde{\tilde{P}}_{jk}\right |^2 \leq -\sum_{j,k} \frac{1- |\sigma_j |\, |\sigma_k| } { \left (1+\sigma_j^2 \right ) \left (1+\sigma_k^2 \right ) }\left | \tilde{\tilde{P}}_{jk}\right |^2\\
&\leq - \frac{1-\delta ^2} { \left (1+\delta ^2 \right )^2}\sum_{j,k }\left | \tilde{\tilde{P}}_{jk}\right |^2  = - \frac {1-\delta ^2}{\left ( 1+\delta ^2 \right )^2} \Bigl |\tilde{\tilde{P}}\Bigl |^2 =- \frac {1-\delta ^2}{\left ( 1+\delta ^2 \right )^2} \big |\tilde{P}\big |^2.
\end{split}
\end{equation*}
Thus we get the first inequality in \eqref{42}. Combining this with Proposition \ref{Pro12}, we can easily obtain the second inequality in \eqref{42}.
\end{proof}

\begin{proposition}\label {Pro14}
Suppose that $R=\omega +\beta \in D_{\delta, \mu}$. Then for any skew-symmetric matrix $Q\in \mathbb{R}^{n \times n},$ we have
\begin{equation}\label{43}
\mathcal{F}(R,Q) = \sum _{j,k=1}^n \frac {1-\sigma_j \sigma_k}{\left ( 1+\sigma_j^2 \right ) \left ( 1+\sigma_k^2 \right )} \left |\tilde{\tilde{Q}}_{jk}\right |^2 .
\end{equation}
\end{proposition}

\begin{proof}
Since $Q$ is skew-symmetric, by Proposition \ref{Pro1}, $Q^{(2)}$ is symmetric. By arguing as in \eqref{36}, we obtain from \eqref{30},
\begin{equation}\label{44}
\begin{split}
&\mathcal{F}(R,Q)= - \left [ \mathcal{G}(R,Q) \right ]^2 +\mathcal{H}(R,Q) =-\left [ \mathcal{G}(R,Q) \right ]^2 \\
&+2 \,{\Tr} \left [ \left ( \frac{ R^{-1} + \left ( R^{-1}  \right )^T}{2}   \right )^{(2)} Q^{(2)}  \right ] +2\, {\Tr} \left [  \left ( \frac{ R^{-1} - \left ( R^{-1}  \right )^T}{2}   \right )^{(2)} Q^{(2)} \right ].
\end{split}
\end{equation}
It follows from \eqref{19} and \eqref{20} that
\begin{equation}\label{45}
\begin{split}
\mathcal{G}(R,Q)&= {\Tr} \left ( \frac{ R^{-1} -\left ( R^{-1} \right ) ^T }{2} Q \right ) ={\Tr} \left ( \omega^{-\frac{1}{2}} C_1 D_3 C_1^*\omega^{-\frac{1}{2}} Q  \right )\\
&={\Tr} \left ( D_3 C_1^*  \omega^{-\frac{1}{2}} Q  \omega^{-\frac{1}{2}} C_1 \right ) = {\Tr} \left ( D_3 \tilde{\tilde{Q}} \right  ) =\sum_{j}\frac{ -{\im} \sigma_j}{1+\sigma_j^2} \tilde{\tilde{Q}}_{jj}   .
\end{split}
\end{equation}
Since $Q$ is skew-symmetric, $\tilde{\tilde{Q}}$ is skew-Hermitian. Hence $\tilde{\tilde{Q}}_{jk} \tilde{\tilde{Q}}_{kj}= -\left | \tilde{\tilde{Q}}_{jk} \right |^2,$ $j,k=1, \ldots,n.$ From these equalities, \eqref{19}, \eqref{20} and Proposition \ref{Pro1}, we obtain
\begin{equation}\label{46}
\begin{split}
2\, {\Tr}& \left [  \left ( \frac{ R^{-1} - \left ( R^{-1}  \right )^T}{2}   \right )^{(2)} Q^{(2)} \right ]=  2 \,{\Tr} \left [  \left (  \omega^{-\frac{1}{2}} C_1 D_3 C_1^*\omega^{-\frac{1}{2}}     \right )^{(2)} Q^{(2)} \right ]\\
&= 2\, {\Tr} \left [ D_3^{(2)} { \tilde{\tilde{Q}}}^{(2)}   \right ] = 2 \sum_{j<k}  \bigl ( D_3 ^{(2)}\bigl )_{jk,jk}{ \tilde{\tilde{Q}}}^{(2)}_{jk,jk}         \\
&= 2 \sum_{j<k} \frac{- \sigma_j \sigma_k    }{\left (1+\sigma_j ^2 \right )  \left (1+\sigma_k ^2 \right )} \left ( \tilde{\tilde{Q}}_{jj} \tilde{\tilde{Q}}_{kk}+\left | \tilde{\tilde{Q}}_{jk} \right |^2  \right )\\
&= \sum_{j\ne k} \frac{- \sigma_j \sigma_k    }{\left  (1+\sigma_j ^2 \right )  \left (1+\sigma_k ^2 \right )} \tilde{\tilde{Q}}_{jj} \tilde{\tilde{Q}}_{kk } +\sum_{j\ne k} \frac{- \sigma_j \sigma_k    }{\left (1+\sigma_j ^2 \right )  \left (1+\sigma_k ^2 \right )} \left | \tilde{\tilde{Q}}_{jk} \right |^2\\
&= \sum_{j, k} \frac{- \sigma_j \sigma_k    }{\left (1+\sigma_j ^2 \right )  \left (1+\sigma_k ^2 \right )} \tilde{\tilde{Q}}_{jj} \tilde{\tilde{Q}}_{kk }+\sum_{j, k} \frac{- \sigma_j \sigma_k    }{\left (1+\sigma_j ^2 \right )  \left (1+\sigma_k ^2 \right )} \left | \tilde{\tilde{Q}}_{jk} \right |^2.
\end{split}
\end{equation}
Combining \eqref{45}, \eqref{46} gives
\begin{equation}\label{47}
2\,{\Tr} \left [ \left ( \frac{ R^{-1} - \left ( R^{-1}  \right )^T}{2}   \right )^{(2)} Q^{(2)}  \right ] = \left [ \mathcal{G}(R,Q)\right ]^2 + \sum_{j,k} \frac{  -\sigma_j \sigma_k}{\left (1+\sigma_j ^2 \right )  \left (1+\sigma_k ^2 \right ) } \left | \tilde{\tilde{Q}}_{jk} \right |^2.
\end{equation}
From the equalities $\tilde{\tilde{Q}}_{jk} \tilde{\tilde{Q}}_{kj}= -\left | \tilde{\tilde{Q}}_{jk} \right |^2$ $(j,k=1, \ldots,n)$, \eqref{19}, \eqref{20} and Proposition \ref{Pro1}, we also get
\begin{equation}\label{48}
\begin{split}
2\,{\Tr} &\left [ \left (  \frac{ R^{-1} + \left ( R^{-1}  \right )^T}{2}   \right )^{(2)} Q^{(2)}  \right ] = 2\, {\Tr} \left [ \left (  \omega^{-\frac{1}{2} } C_1 D_2 C_1^* \omega^{-\frac{1}{2} }   \right )^{(2 )}  Q^{(2)}\right ]\\
&=2 \,{\Tr} \left [D_2 ^{ (2) } {\tilde{\tilde{Q}}}^{(2)}   \right ] = 2 \sum_{j<k} \frac{ 1}{\left (1+\sigma_j ^2 \right )  \left (1+\sigma_k ^2 \right ) } \left (\tilde{\tilde{Q}}_{jj} \tilde{\tilde{Q}}_{kk}+ \left | \tilde{\tilde{Q}}_{jk}   \right |^2 \right )\\
&=\sum_{j \ne k} \frac{1}{\left (1+\sigma_j ^2 \right )  \left (1+\sigma_k ^2 \right ) } \tilde{\tilde{Q}}_{jj} \tilde{\tilde{Q}}_{kk}+ \sum_{j \ne k} \frac{ 1}{\left (1+\sigma_j ^2 \right )  \left (1+\sigma_k ^2 \right ) } \left | \tilde{\tilde{Q}}_{jk}\right |^2\\
&=\sum_{j ,k} \frac{1}{\left (1+\sigma_j ^2 \right )  \left (1+\sigma_k ^2 \right ) } \tilde{\tilde{Q}}_{jj} \tilde{\tilde{Q}}_{kk}+ \sum_{j ,k} \frac{ 1}{\left (1+\sigma_j ^2 \right )  \left (1+\sigma_k ^2 \right ) } \left | \tilde{\tilde{Q}}_{jk}\right |^2\\
&= \Biggl ( \sum_{j} \frac{\tilde{\tilde{Q}}_{jj}}{1+\sigma_j^2 }\Biggl )^2+\sum_{j,k} \frac{1}{ \left ( 1+\sigma_j^2\right ) \left ( 1+\sigma_k^2 \right ) } \left | \tilde{\tilde{Q}}_{jk}\right |^2.
\end{split}
\end{equation}
 
\noindent Obviously, $ {\Tr} \left ( \dfrac{ R^{-1} + \left  ( R^{-1} \right ) ^T }{2} Q    \right )=0.$ It follows from this, \eqref{19} and \eqref{20} that
$$ {\Tr} \left ( \omega^{-\frac{1}{2} }C_1 D_2 C_1^* \omega^{-\frac{1}{2}}Q   \right )={\Tr} \left ( D_2 \tilde{\tilde{Q}} \right )=\sum_{j} \frac{\tilde{\tilde{Q}}_{jj}}{1+\sigma_j^2 } =0. $$
Combining this and \eqref{48} gives
\begin{equation}\label{49}
2\,{\Tr} \left [ \left (  \frac{ R^{-1} + \left ( R^{-1}  \right )^T}{2}   \right )^{(2)} Q^{(2)}  \right ]=\sum_{j,k} \frac{1}{ \left (1+\sigma_j^2 \right ) \left (1+\sigma_k^2 \right ) } \left | \tilde{\tilde{Q}}_{jk}\right |^2.
\end{equation}
The proof is straightforward from \eqref{44}, \eqref{47} and \eqref{49}.
\end{proof}

\begin{corollary} \label{Cor4}
Suppose that $R=\omega +\beta \in D_{\delta, \mu}$. Then for any skew-symmetric matrix $Q \in \mathbb{R}^{n \times n},$ we have
\begin{equation}\label{50}
\mathcal{F}(R,Q) \leq \big | \tilde{Q} \big |^2.
\end{equation}
\end{corollary}

\begin{proof}
Note that
$$ \frac{1-\sigma_j \sigma_k } { \left ( 1+\sigma_j ^2 \right ) \left ( 1+\sigma_k^2 \right )} \leq \frac{1 +|\sigma_j| |\sigma_k|} { \left ( 1+\sigma_j ^2 \right ) \left ( 1+\sigma_k^2 \right ) } \leq 1, \  j,k =1, \ldots,n .$$
From this, \eqref{33} and \eqref{43}, we obtain that
$$ \mathcal{F}(R,Q) \leq \sum_{j,k} \left | \tilde{\tilde{Q}}_{jk} \right |^2 = \left | \tilde{\tilde{Q}}\right |^2 = \big | \tilde{Q}\big |^2. $$
This completes the proof.
\end{proof}

\begin{proposition} \label{Pro15}
Suppose that $R=\omega +\beta \in D_{\delta, \mu}.$ Then for any symmetric matrix $P\in \mathbb{R}^{n \times n}$ and any skew-symmetric matrix $Q\in \mathbb{R}^{n \times n},$ we have
\begin{equation}\label{51}
|\mathcal{L}(R,P,Q)| \leq \frac{2n \delta}{1+\delta^2 } \big | \tilde{P}\big | \big |\tilde{Q}\big |.
\end{equation}
\end{proposition}

\begin{proof}
By \eqref{17} and \eqref{31}, we have
\begin{equation*}
\begin{split}
\mathcal{L}(R,P,&Q) = -2\, {\Tr} \left [  \frac{\left ( R^{-1} \right ) - \left ( R^{-1}\right ) ^T }{2} P \frac{  \left ( R^{-1} \right )+  \left ( R^{-1}\right ) ^T }{2} Q    \right ]\\
&= -2\, {\Tr} \left [ \left( \omega^{-\frac{1}{2}} (-\sigma) \left ( E-\sigma ^2 \right )^{-1} \omega^{-\frac{1}{2}} \right ) P \left( \omega^{-\frac{1}{2}} \left ( E-\sigma ^2 \right )^{-1} \omega^{-\frac{1}{2}} \right )  Q      \right ]\\
&=2\, {\Tr} \left [ \sigma \left ( E-\sigma ^2 \right )^{-1} \left ( \omega^{-\frac{1}{2}} P \omega^{-\frac{1}{2}}  \right ) \left ( E-\sigma ^2 \right )^{-1}  \left ( \omega^{-\frac{1}{2}} Q \omega^{-\frac{1}{2}}    \right )  \right ]\\
&= 2\, {\Tr} \left [\sigma \left ( E- \sigma ^2 \right )^{-1} \tilde{P} \left ( E-\sigma ^2 \right )^{-1} \tilde{Q}  \right ].
\end{split}
\end{equation*}
Hence
\begin{equation} \label{52}
 | \mathcal{L} (R,P,Q ) | \leq 2 \left | \sigma \left ( E-\sigma^2 \right )^{-1}\right | \big | (E-\sigma^2 )^{-1} \big |  \big | \tilde{P}\big | \big |\tilde{Q}\big | .
\end{equation}
From \eqref{10}, \eqref{12} and Proposition \ref{Pro4}, we can easily obtain
\begin{equation*}
\begin{split}
\left | \sigma  \left ( E-\sigma^2 \right )^{-1}   \right |& =\left | D_1  \left ( E- D_1 ^2 \right )^{-1}   \right |= \left (\sum_j  \frac{ \sigma_j^2}{\left ( 1+\sigma_j^2   \right ) ^2}  \right )^{1/2} \leq \frac {\sqrt{n}\, \delta}{1+\delta^2 },\\
\left | (E-\sigma^2 )^{-1} \right | &=  \left | \left ( E-D_1^2 \right )^{-1} \right |=\left ( \sum_j \frac{1} {\left ( 1+\sigma_j^2   \right ) ^2}  \right )^{1/2} \leq \sqrt{n}.
\end{split}
\end{equation*}
Combining these estimates with \eqref{52}, we get the desired estimate \eqref{51}.
\end{proof}

In the next theorem we will give an upper estimate for second-order differentials of the function $F(R).$
\begin{theorem}\label{Theo1}
Suppose that $R=\omega +\beta \in D_{\delta, \mu}.$ Then for any matrix $M = P + Q,$ where $P\in \mathbb{R}^{n \times n}$ is symmetric and $Q \in \mathbb{R}^{n \times n}$ is skew-symmetric, we have
\begin{equation}\label{53}
\mathcal{F}(R,M)  \leq - (1-\eta)  \frac {1-\delta ^2 }{\left ( 1+\delta ^2 \right )^2}\big | \tilde{P} \big |^2 + \left (1 + \frac{ 4 n^2 \delta ^2}{ \eta \left (1-\delta ^2 \right )} \right ) \big | \tilde{Q}\big |^2 ,
\end{equation}
for any constant $\eta \in (0, 1],$ where $\tilde{P}=\omega ^{-\frac {1}{2}}P\omega ^{-\frac {1}{2}}, \tilde{Q}=\omega ^{-\frac {1}{2}}Q\omega ^{-\frac {1}{2}}.$
\end{theorem}

\begin{proof}
From \eqref{26}, \eqref{42}, \eqref{50} and \eqref{51}, we have
\begin{equation*}
\begin{split}
\mathcal{F}(R,M)&=  \mathcal{F}(R,P) + \mathcal{F}(R,Q) + 2 \mathcal{L}(R,P,Q)\\
& \leq - \frac{1-\delta^2 } {\left ( 1+\delta ^2 \right )^2 }\big | \tilde{P}\big | ^2 +\big | \tilde{Q}\big |^2 +\frac{4n \delta}{1+\delta^2} \big |\tilde{P}\big | \big |\tilde{Q}\big |.
\end{split}
\end{equation*}
By using Cauchy's inequality, we have for any positive constant $\eta \in (0,1],$
\begin{equation*}
\frac{4n \delta}{1+\delta^2} \big |\tilde{P}\big | \big |\tilde{Q}\big |  \leq  \frac{\eta \left (1-\delta ^2 \right ) }{ \left (1+\delta ^2 \right ) ^2 } \big |\tilde{P}\big | ^2 + \frac{ 4 n^2 \delta ^2}{ \eta \left (1-\delta ^2 \right )} \big |\tilde{Q} \big |^2.   
\end{equation*}
Combining these estimates, we obtain the estimate \eqref{53}. The proof is completed.
\end{proof}

\subsection{The $d-$concavity of the function $F(R)$}

\begin{theorem}\label{Theo2}
For any matrices $R^{(0)} = \omega^{(0)} + \beta^{(0)} = \left [ R^{(0)}_{ij}\right ], R^{(1)} = \omega^{(1)} + \beta^{(1)} = \left [ R^{(1)}_{ij}\right ]$ from the set $D_{\delta,\mu},$ we have 
\begin{equation} \label{54}
\begin{split}
F \bigl ( R^{(1)}\bigl ) - F \bigl ( R^{(0)} \bigl )  &\leq  \sum _{i,j=1}^n\frac {\partial F \bigl (R^{(0)} \bigl )}{\partial R_{ij}}\Bigl (R^{(1)}_{ij} - R^{(0)}_{ij}\Bigl ) \\
&+\frac{1}{2} \left (1 + \frac{ 4 n^2 \delta ^2}{ 1-\delta ^2 } \right ) \big (\lambda_{\min}\big ( \omega^{(s)} \big ) \big )^{-2} \bigl |\beta^{(1)}- \beta^{(0)}\bigl |^2,
\end{split}
\end{equation}
where $\omega^{(s)} \equiv (1-s) \omega^{(0)} + s\omega^{(1)}$ for some constant $s \in (0,1)$. 
\end{theorem}

\begin{proof}
We set for all $t \in [0,1],$
\begin{equation*}
    g(t) := F \bigl ( (1-t) R^{(0)} + tR^{(1)}\bigl )=F \bigl (R^{(t)}\bigl ),
\end{equation*}
where $R^{(t)}\equiv (1-t) R^{(0)} + tR^{(1)}=\omega^{(t)}+\beta^{(t)},$ $\omega^{(t)}= (1-t) \omega^{(0)} +t \omega^{(1)},$ $\beta^{(t)}= (1-t) \beta^{(0)}+ t \beta^{(1)}.$  Since $D_{\delta, \mu}$ is convex, we infer that $R^{(t)}\in D_{\delta, \mu}.$

By the Taylor expansion, we have for some constant $s \in (0,1),$
\begin{equation}\label{55}
F \big ( R^{(1)}\big ) - F \big ( R^{(0)}\big ) = g(1) - g(0) = g'(0) + \frac {1}{2}g''(s). 
\end{equation}

\noindent By computation, we have for all $t \in (0,1),$
\begin{equation*}
\begin{split}
g'(t) &= \sum_{i,j} \frac {\partial F \bigl ( R^{(t)} \bigl )}{\partial R _{ij}} \Bigl ( R^{(1)}_{ij}- R^{(0)}_{ij}\Bigl ),\\
    g''(t) &= \sum_{i,j,k,\ell} \frac{\partial^2 F \bigl ( R^{(t)}\bigl )}{\partial R _{ij} \partial R _{k \ell}} \left ( R^{(1)}_{ij} - R^{(0)}_{ij}\right ) \left ( R^{(1)}_{k\ell} - R^{(0)}_{k\ell}\right ) = \mathcal {F}\bigl ( R^{(t)}, R^{(1)} - R^{(0)} \bigl ),
\end{split}
\end{equation*}
where the function $\mathcal{F}$ is defined by \eqref{25}. Hence
\begin{equation}\label{56}
g'(0) = \sum_{i,j} \frac {\partial F\bigl ( R^{(0)} \bigl )}{\partial R _{ij}} \Bigl ( R^{(1)}_{ij}- R^{(0)}_{ij}\Bigl ).
\end{equation}

\noindent Moreover, by applying Theorem \ref{Theo1} with $R=R^{(s)}=\omega^{(s)}+\beta^{(s)}$, $M = R^{(1)} - R^{(0)} = \bigl ( \omega^{(1)} -\omega^{(0)} \bigl )+ \bigl ( \beta^{(1)}  - \beta^{(0)} \bigl ) \equiv P + Q $ and $\eta=1$, we obtain 
\begin{equation*} 
\begin{split}
g''(s) =\mathcal {F}\bigl (R^{(s)}, R^{(1)} - R^{(0)} \bigl ) &\leq \left (1 + \frac{ 4 n^2 \delta ^2}{ 1-\delta ^2 } \right )\Bigl | {\big (\omega^{(s)}\big )} ^{-\frac{1}{2}} \bigl ( \beta^{(1)}- \beta^{(0)}\bigl )  {\big (\omega^{(s)}\big )}^{-\frac{1}{2}} \Bigl |^2 \\
&\leq \left (1 + \frac{ 4 n^2 \delta ^2}{ 1-\delta ^2 } \right ) \big (\lambda_{\min}\big ( \omega^{(s)} \big ) \big )^{-2} \bigl |\beta^{(1)}- \beta^{(0)}\bigl |^2,
\end{split}
\end{equation*}
where the last inequality is by Proposition \ref{Pro12}. Combining this estimate with \eqref{55} and \eqref{56}, we arrive at the estimate \eqref{54}.
\end{proof}

Now, we obtain the following theorem on $d$-concavity in the set $D_{\delta,\mu}$ for the Monge-Amp\`{e}re type function $F(R)$.
\begin{theorem}\label{Theo3}
The function $F(R) = \log (\det R)$ is $d$-concave in the set $D_{\delta,\mu},$ where $d=2n\delta ^2 \left (1 + \dfrac{ 4 n^2 \delta ^2}{ 1-\delta ^2 } \right )$, depending only on $\delta$ and $n$. That means, for any matrices $R^{(0)} = \omega^{(0)} + \beta^{(0)} = \left [ R^{(0)}_{ij}\right ], R^{(1)} = \omega^{(1)} + \beta^{(1)} = \left [ R^{(1)}_{ij}\right ]$ from $D_{\delta,\mu},$ we have
\begin{equation}\label{57}
 F\big (R^{(1)}\big ) - F\big ( R^{(0)}\big )  \leq  \sum _{i,j=1}^n\frac {\partial F \big ( R^{(0)}\big )}{\partial R_{ij}} \left ( R^{(1)}_{ij} - R^{(0)}_{ij}\right )  
  + d.
\end{equation}
\end{theorem}

\begin{proof}
By the assumptions and the definition of $D_{\delta, \mu}$ in \eqref{5}, we have
$$ \big |\beta^{(1)}- \beta^{(0)}\big |^2 \leq n\, \big \|\beta^{(1)}- \beta^{(0)}\big \|^2 \leq 2n  \left ( \big \| \beta^{(0)}\big \|^2 +  \big \| \beta^{(1)}\big \|^2 \right ) \leq 4 n\mu ^2, $$
and
$$ \delta \lambda_{\min}\big ((1-s) \omega^{(0)} +s \omega^{(1)}\big ) \geq \mu,\ \forall s \in [0,1].$$
From these estimates and \eqref{54}, we can easily obtain the desired estimate \eqref{57}.
\end{proof}

\section{Comparison principle for nonsymmetric Monge-Amp\`{e}re type equations }

In this section, we shall establish the comparison principle for the Monge-Amp\`{e}re type equation \eqref{1} in the case that it is $\delta-$elliptic, $0 \leq \delta <1$ with respect to compared functions. Consider the following operator associated to the equation \eqref{1},
\begin{equation*}
G[u](x)\equiv \log \det \left [ D^2u -A(x,u,Du)-B(x,u,Du) \right ]-\log f(x,u,Du  ), x \in \Omega.
\end{equation*}

\begin{theorem}\label{Theo4}
Let $u(x),v(x) \in  C^2(\overline {\Omega})$ satisfying $G[u](x) \leq G[v](x)$ in $\Omega$, $u \geq v$ on $\partial \Omega,$ where $A, B, f$ are in $ C^1 (\overline{\Omega} \times \mathbb{R} \times \mathbb{R}^n)$ and $f >0$ on $\overline{\Omega} \times \mathbb{R} \times \mathbb{R}^n.$ Suppose that the following conditions are satisfied for some nonnegative constants $\delta, \alpha_1, \beta_1$, $0 \leq \delta <1$ and for all $x \in \overline{\Omega}$, $z \in \mathbb{R}$, $p\in \mathbb{R}^n,$\\
\indent {\rm (i)} $ \lambda_{\min} (\omega(x,u)) >0 , \  \lambda_{\min} (\omega(x,v)) >0 ;$\\
\indent {\rm (ii)} $ \delta   \min \{\lambda_{\min} (\omega(x,u)),  \lambda_{\min} (\omega(x,v)) \}  \geq  \mu (B) ;$\\
\indent {\rm (iii)} $\lambda_{\min}(D_z A(x,z,p) ) \geq (-\alpha_1)  \min \{  \lambda_{\min} (\omega(x,u)),  \lambda_{\min} (\omega(x,v)) \}$;\\
\indent {\rm (iv)} $ \beta_1 \min\{  \lambda_{\min} (\omega(x,u)) ,  \lambda_{\min} (\omega(x,v)) \} \geq \mu  (D_z B)$;\\
\indent {\rm (v)} $\underset{\overline{\Omega} \times \mathbb{R} \times \mathbb{R}^n}{\inf}\,\left ( \frac {D_z f }{f} \right ) \geq  n \left ( \alpha_1 + \dfrac {\delta}{1+\delta^2}\beta_1 \right ),$ 

where the quantities $\mu (B), \mu (D_z B)$ are defined as in \eqref{6}. 

Then we have that either $u>v$ or $u\equiv v$ in $\Omega.$
\end{theorem}

\begin{proof}
For all $x \in \overline{\Omega}$ and for all $t \in [0,1],$ we set
$$ w(x)=v(x)-u(x) ,$$
$$ u^{(t)}(x)=(1-t)u(x)+tv(x),$$
and
\begin{equation*}
\begin{split}
R^{(0)}(x)&=D^2u(x)-A(x,u(x),Du(x))-B(x,u(x),Du(x)), \\
R^{(1)}(x)&=D^2v(x)-A(x,v(x),Dv(x))-B(x,v(x),Dv(x)), \\
R^{(t )}(x)&= (1-t )R^{(0)}(x) + tR^{(1)}(x), \\
 \omega^{(0)}(x)&=D^2u(x)-A(x,u(x),Du(x)), \\
 \omega^{(1)}(x)&=D^2v(x)-A(x,v(x),Dv(x)), \\
\omega^{(t )}(x)&= (1-t )\omega^{(0)}(x)+t \omega^{(1)}(x).
\end{split}
\end{equation*}
Set
 $$g(t,x) \equiv \log \det \big ((1-t)R^{(0)}(x)+tR^{(1)}(x) \big ) =\log \det \big (R^{(t)}(x) \big ) .$$
Then by the mean value Theorem and \eqref{23}, we have 
\begin{equation}\label{58}
\begin{split}
\log \det  \big ( R^{(1)}(x) \big ) &- \log \det \big (R^{(0)}(x) \big ) =g(1,x)-g(0,x)\\
&=g_t'(s,x)=\sum_{i,j=1}^n \big (R^{(s)}(x) \big )^{-1}_{ji}\big (R^{(1)}_{ij}(x) - R^{(0)}_{ij}(x) \big ),
\end{split}
\end{equation}
where $s \in (0,1)$ is the constant depending on $x.$

\noindent Set
\begin{equation*}
\begin{split}
h(t,x)= \sum_{i,j=1}^n &\big (R^{(s)}(x) \big )^{-1}_{ji} \bigl [D_{ij}u^{(t)}(x) - A_{ij}\big (x,u^{(t)}(x), Du^{(t)}(x) \big ) \\
&- B_{ij}\big (x,u^{(t)}(x), Du^{(t)}(x) \big ) \bigl ] -\log f \big (x,u^{(t)}(x), Du^{(t)}(x) \big ).
\end{split}
\end{equation*}
Then by the mean value Theorem and \eqref{58}, we obtain
\begin{equation*}
\begin{split}
&G[v](x)-G[u](x)=h(1,x)-h(0,x)=h_t'(\tau ,x)\\
 &=\sum_{i,j=1}^n \big (R^{(s)}(x) \big )^{-1}_{ji} \biggl [ ( D_{ij}v(x)-D_{ij}u(x)  )\\
&-\sum_{k=1}^n   ( D_{p_k}A_{ij} +D_{p_k} B_{ij}) \big (x,u^{(\tau)}(x), Du^{(\tau)}(x) \big ) (D_k v(x)-D_k u(x))\\
&- (D_z A_{ij}+D_z B_{ij})\big (x, u^{(\tau)}(x), Du^{(\tau)}(x)\big ) (v(x)-u(x))  \biggl ]\\
&-\frac{1}{f \big (x, u^{(\tau)}(x), Du^{(\tau)}(x) \big )} \sum_{k=1}^n D_{p_k} f \big (x,u^{(\tau)}(x), Du^{(\tau)}(x) \big )  (D_k v(x)-D_k u(x))\\
&-\frac {1}{f \big (x,u^{(\tau)}(x) , Du^{(\tau)}(x) \big )}D_z f \big (x,u^{(\tau)}(x), Du^{(\tau)}(x) \big ) (v(x)-u(x)),
\end{split}
\end{equation*}
where $\tau \in (0,1)$ is the constant depending on $x$ and $s.$

\noindent Consequently,
\begin{equation} \label{59}
 G[v](x)-G[u](x)=a^{ij}(x)D_{ij} w(x)+b^k(x)D_kw(x) +c(x)w(x) ,
\end{equation}
where
\begin{equation}\label{60}
\begin{split}
 a^{ij}(x)&= \frac {\big (R^{(s)}(x) \big )^{-1}_{ij} + \big ( R^{(s)}(x) \big )^{-1}_{ji}}{2} ,\\
 b^k(x)&=-\sum_{i,j=1}^n \big (R^{(s)}(x) \big )^{-1}_{ji} \big [ (D_{p_k} A_{ij} + D_{p_k} B_{ij}) \big (x, u^{(\tau)}(x), Du^{(\tau)}(x) \big ) \big ] \\
& -\frac {1}{f \big ( x,u^{(\tau)}(x) , Du^{(\tau)}(x) \big )} D_{p_k} f \big (x,u^{(\tau)}(x), Du^{(\tau)}(x) \big ),\\
c(x)&=-\sum_{i,j=1}^n \big ( R^{(s)}(x) \big )^{-1}_{ji} \big [ (D_z A_{ij}+D_z B_{ij}) \big (x, u^{(\tau)}(x), Du^{(\tau)}(x) \big )  \big ]\\
&-\frac {1}{f \big (x, u^{(\tau)}(x), Du^{(\tau)}(x) \big )} D_z f \big ( x, u^{(\tau)}(x), Du^{(\tau)}(x) \big ).
\end{split}
\end{equation}
Consider the second order linear partial differential operator $L$ given by
\begin{equation}\label{61}
L:= a^{ij}(x)D_{ij} +b^k(x)D_k +c(x), 
\end{equation}
where the coefficients $a^{ij}, b^k, c$ are defined by \eqref{60}. We have the following claims.

\noindent {\bf Claim 1.} The operator $L$ is uniformly elliptic; that is, there exists positive constants $\lambda, \Lambda$ such that
\begin{equation}\label{62}
\lambda |\xi|^2 \leq a^{ij}(x) \xi_i \xi_i \leq \Lambda |\xi|^2, \ \forall x \in \overline{\Omega}, \forall \xi \in \mathbb{R}^n.
\end{equation}

Indeed, it follows from conditions (i), (ii) that $R^{(0)}(x), R^{(1)}(x)$ are in the set $D_{\delta, \mu(B)}$, so is $R^{(s)}(x)$. Also from (i) and our regularity assumptions for $A,B$ and $u$, we infer that there exists positive constants $\lambda_0, \Lambda_0$ such that
\begin{equation*}
\lambda_0 E \leq \omega (x,u) \leq \Lambda_0 E, \ \  \lambda_0 E \leq \omega (x,v)  \leq \Lambda_0 E,\ \forall x \in \overline{\Omega},
\end{equation*}
where $E$ is the unit matrix of order $n$. It follows that
\begin{equation}\label{63}
   \lambda_0 E \leq \omega^{(s)}(x)   \leq \Lambda_0 E, \ \forall x \in \overline{\Omega}.
\end{equation}
Therefore
$$  \frac{1}{\Lambda_0} E \leq \big (\omega^{(s)}(x) \big )^{-1}  \leq \frac{1}{\lambda_0} E, \ \forall x \in \overline{\Omega}.$$

\noindent Moreover, by Proposition \ref{Pro4} and Corollary \ref{Cor1}, one can easily show that
\begin{equation*}
 \frac{1}{1+\delta^2} \Bigl ( \big (\omega^{(s)}(x)\big )^{-1}    \xi, \xi \Bigl ) \leq  \left ( H(x) \xi, \xi \right )    \leq  \Bigl ( \big (\omega^{(s)}(x)\big )^{-1} \xi, \xi \Bigl ) , \ \forall x \in \overline{\Omega},
\end{equation*}
where
$$ H(x):=\frac {\big (R^{(s)}(x) \big )^{-1}+ \Bigl ( \big ( R^{(s)}(x)\big )^{-1}\Bigl )^T}{2}.  $$
Then \eqref{62} follows from the above estimates by taking $\lambda=\dfrac{1}{(1+\delta^2) \Lambda_0}$ and $\Lambda= \dfrac{1}{\lambda_0}.$

\noindent {\bf Claim 2.} The coefficients $b^k(x), c(x)$ are bounded in $\overline{\Omega}$.

This claim easily follows from the fact that the set $ \big \{ \big (x, u^{(\tau)}(x), Du^{(\tau)}(x) \big ) \big \}$ is bounded in $\overline{\Omega } \times \mathbb R \times \mathbb {R}^n$ and 
$$ \det R^{(s)}(x) \geq \det \omega^{(s)}(x) \geq \lambda_0^n>0, \ \forall x \in \overline{\Omega}, $$
which holds by Proposition \ref{Pro3} and \eqref{63}.

\noindent {\bf Claim 3.} The coefficient $c(x) \leq 0$ for all $ x\in \overline {\Omega}$.

{\bf Claim 3} follows from \eqref{60}, condition (v) and the two following inequalities
\begin{equation}\label{64}
 -\sum_{i,j=1}^n \big ( R^{(s)}(x) \big )^{-1}_{ji} D_z A_{ij}\big ( x, u^{(\tau)}(x), Du^{(\tau)}(x) \big ) \leq  n \alpha_1, \ \forall x \in \overline{\Omega},
\end{equation}
\begin{equation}\label{65}
 -\sum_{i,j=1}^n \big ( R^{(s)}(x) \big )^{-1}_{ji} D_z B_{ij}\big ( x, u^{(\tau)}(x), Du^{(\tau)}(x) \big ) \leq n \frac{\delta}{1+\delta^2} \beta_1, \ \forall x \in \overline{\Omega}.
\end{equation}
So it remains to prove \eqref{64} and \eqref{65}. 

Since $D_z A$ is symmetric and $H(x)$ is positive definite, we have
\begin{equation}\label{66}
  \sum_{i,j=1}^n \big ( R^{(s)}(x) \big )^{-1}_{ji} D_z A_{ij}={\Tr} [ (D_z A) (H(x)) ] \geq \lambda_{\min} (D_z A)\, {\Tr} H(x).
\end{equation}
Given any point $x \in \overline{\Omega}.$ Assume that $\lambda_{\min} \big (D_z A \big ( x, u^{(\tau)}(x), Du^{(\tau)}(x) \big )\big ) \geq 0$ at this point. Then by \eqref{66}, the left hand side of \eqref{64} is nonpositive and thus \eqref{64} follows. Assume the contrary, that $\lambda_{\min} \big (D_z A \big ( x, u^{(\tau)}(x), Du^{(\tau)}(x) \big )\big )< 0.$ Then \eqref{64} follows from \eqref{66} and the following estimates 
$$ {\Tr}H(x) \leq {\Tr} \Bigl [ \big ( \omega^{(s) }(x)\big )^{-1}\Bigl ] \leq  \frac{n}{\lambda_{\min} \big ( \omega^{(s)}(x) \big )},  $$
\begin{equation*}
\begin{split}
(-\alpha_1 ) \lambda_{\min} (\omega^{(s)}(x))  &\leq  (-\alpha_1) \min \{ \lambda_{\min} (\omega(x,u) ) ,  \lambda_{\min} (\omega(x,v) ) \}\\
& \leq \lambda_{\min} \big (D_z A \big ( x, u^{(\tau)}(x), Du^{(\tau)}(x) \big )\big ),
\end{split}
\end{equation*}
which are inferred from Corollary \ref{Cor2} and condition (iii), respectively. Thus \eqref{64} is proved.

We now prove \eqref{65}. Set
$$ K(x):=  \frac {\big ( R^{(s)}(x)\big )^{-1}-\Bigl ( \big (R^{(s)}(x)\big )^{-1}\Bigl )^T}{2} . $$
By Proposition \ref{Pro4} and Corollary \ref{Cor1}, one can easily show that
\begin{equation*}
 \| K(x) \|  \leq \frac {\delta}{1+\delta^2}\frac {1}{\lambda_{\min}\big ( \omega^{(s)}(x) \big )}, \ \forall x \in \overline{\Omega},
\end{equation*}
and, by condition (iv),
\begin{equation*}
\mu  (D_z B) \leq \beta_1 \min \{ \lambda_{\min} (\omega(x,u) ) ,  \lambda_{\min} (\omega(x,v) ) \}\leq \beta_1 \lambda_{\min}\big ( \omega^{(s)}(x)\big ) , \ \forall x \in \overline{\Omega}.
\end{equation*}
From these estimates and the following inequality:
\begin{equation*}
{\Tr} ( MN ) \leq |M| |N| \leq n \|M \| \|N\|,\,\,\, \text{for all}\,\,\, M,N \in \mathbb{R}^{n \times n},
\end{equation*}
we obtain for all $x \in  \overline{\Omega},$
\begin{equation*}
\begin{split}
  -\sum_{i,j=1}^n \big (R^{(s)}(x) \big )^{-1}_{ji}&D_z B_{ij}(x, u^{(\tau)}(x), Du^{(\tau)}(x))\\
&={\Tr} \big [ (K(x) ) (-D_z B(x, u^{(\tau)}(x), Du^{(\tau)}(x))) \big ]\\
&\leq n \frac{\delta}{1+\delta^2} \frac{\mu (D_z B)}{\lambda_{\min}(\omega^{(s)}(x))}\leq n \frac{\delta}{1+\delta^2} \beta_1.
\end{split}
\end{equation*}
Thus \eqref{65} is proved.

To complete the proof of this theorem, we note that, if $G[u] \leq G[v] \  \text{in} \ \Omega,  u\geq v \ \text{on} \ \partial \Omega$ then, by \eqref{59} and \eqref{61}, $Lw \geq 0 \ \text{in}\ \Omega , w \leq 0 \    \text{on} \ \partial \Omega.$ By {\bf Claims 1, 2, 3}, we can apply the strong maximum principle of E. Hopf (Theorem 3.5, \cite{2}) to obtain the conclusion of Theorem \ref{Theo4}.
\end{proof}

\begin{corollary} 
Under the assumptions of Theorem \ref{Theo4}, where $G[u]<G[v]$ in $\Omega$, $u = v$ on $\partial \Omega$, $\partial \Omega \in C^2,$ we have the following strict inequalities
\begin{align*}
u&>v, \ \text{in} \ \Omega,\\
\frac{\partial u}{\partial \nu} &> \frac{\partial v}{\partial \nu},\ \text{on} \ \partial \Omega,
\end{align*}
where $\nu$ is the unit inner normal to $\partial \Omega.$
\end{corollary}

\begin{proof}
Set $w=v-u.$ Following the proof of Theorem \ref{Theo4}, we have
$$ Lw>0 \ \text {in}\ \Omega,\ w=0 \ \text{on} \ \partial \Omega, $$
where $L$ is the uniformly elliptic operator defined by \eqref{61}. Further, by Theorem \ref{Theo4}, we have that $w<0$ in $\Omega.$ Hence, we can apply the Hopf's lemma (Theorem 3.4, \cite{2}) to obtain
$$  \frac{\partial w}{\partial \nu} <0 , \ \text{on}\ \partial \Omega.$$
The proof is completed.
\end{proof}


\end{document}